\documentclass[12pt,leqno]{amsart}
\usepackage{color,amsmath,amsfonts,amssymb,amscd,amsthm,amsbsy,upref, siunitx,units}
\usepackage{tikz}
\usepackage{tkz-euclide}
\usepackage{bbm}
\usepackage{multicol}
\usepackage{enumitem}
\usepackage[utf8]{inputenc}
\usepackage[T1]{fontenc}
\usepackage{tikz,pgfplots}
\usepackage{graphicx,bigints}

\numberwithin{equation}{section}
\def\hangbox to #1 #2{\vskip3pt\hangindent #1\noindent \hbox to #1{#2}$\!\!$}

\newtheorem{thm}{Theorem}[section]

\newtheorem{lem}[thm]{Lemma}
\newtheorem{cor}[thm]{Corollary}
\newtheorem{prob}[thm]{Problem}

\newtheorem{prop}[thm]{Proposition}
\theoremstyle{definition}

\theoremstyle{remark}

\def\N{{\mathbb N}}
\def\R{{\mathbb R}}

\def\E{{\mathbb E}}
\def\sfrac#1#2{\kern.1em\raise.5ex\hbox{$#1$}
        \kern-.1em/\kern-.05em\lower.25ex\hbox{$#2$}}
\def\vp{\varepsilon}

\newcommand{\HH}{\mathbb H}

\newcommand{\fw}{\text{\fw}}

\newcommand{\spn}{{\rm span}}

\begin{document}
\allowdisplaybreaks
\title{Discretizing $L_p$ norms and frame theory}

\author{ Daniel Freeman}\author{ Dorsa Ghoreishi}
\address{Department of Mathematics and Statistics\\
St Louis University\\
St Louis MO 63103  USA
} \email{daniel.freeman@slu.edu, dorsa.ghoreishi@slu.edu}

\begin{abstract}
    Given an $N$-dimensional subspace $X^N$ of $L_p([0,1])$, we consider the problem of choosing $M$-sampling points which may be used to discretely approximate the $L_p$ norm on the subspace.   We are particularly interested in knowing when the number of sampling points $M$ can be chosen on the order of the dimension $N$.  For the case $p=2$ it is known that $M$ may always be chosen on the order of $N$ as long as the subspace $X^N$ satisfies a natural $L_\infty$ bound, and for the case $p=\infty$ there are examples where $M$ may not be chosen on the order of $N$.   We show for all $1\leq p<2$ that there exist classes of subspaces of $L_p([0,1])$ which satisfy the  $L_\infty$ bound, but where the number of sampling points $M$ cannot be chosen on the order of $N$.
    We show as well that the problem of discretizing the $L_p$ norm of subspaces is directly connected with frame theory.
    In particular, we prove that discretizing a continuous frame to obtain a discrete frame which does stable phase retrieval requires discretizing both the $L_2$ norm and the $L_1$ norm on the range of the analysis operator of the continuous frame.

\end{abstract}

\thanks{2010 \textit{Mathematics Subject Classification}: 46B03, 46B15, 46E30, 42C15}

\thanks{The first author was supported by grant 706481 from the Simons Foundation.}

\maketitle

\section{Introduction}

Many problems in applied math, physics, and engineering are stated in terms of some continuous structure, but only become computationally feasible when appropriately discretized.  Some examples of this are using numerical integration to estimate the integral of a function over a measure space, or using a fast Fourier transform as a discretization of the Fourier transform.  We will be considering the problem of discretizing the $L_p$ norm on finite dimensional subspaces of $L_p(\Omega)$ where $1\leq p<\infty$ and $\Omega$ is a probability space. We let $X^N\subseteq L_p(\Omega)$ be an $N$-dimensional subspace and note that 
\begin{equation}\label{E:int}
    \|f\|^p_p=\int_\Omega |f(t)|^p \, d\mu(t)\hspace{.5cm}\textrm{ for all } 
    f\in X^N.
\end{equation}
Given some $A<1<B$, we are interested in the problem of discretizing \eqref{E:int} by choosing sampling points 
 $(t_j)_{j=1}^M\subseteq \Omega$  such that 
\begin{equation}\label{E:discI}
    A\|f\|^p_p\leq \frac{1}{M}\sum_{j=1}^M |f(t_j)|^p\leq B\|f\|^p_p\hspace{.5cm}\textrm{ for all } 
    f\in X^N.
\end{equation}
As $X^N$ is finite dimensional, the law of large numbers gives that if $M\in\N$ is large enough and we randomly choose sampling points $(t_j)_{j=1}^M\subseteq \Omega$ then \eqref{E:discI} is satisfied with high probability. However, this does not give any bounds on how large $M$ must be for a given $X^N\subseteq L_p(\Omega)$,
and as one of our motivations is to make a problem computationally feasible, it is of fundamental importance to obtain good bounds on the number of sampling points $M$.  This problem was first considered by Marcinkiewics and later by Zygmund for the case of discretizing the $L_p$ norm on spaces of trigonometric polynomials \cite{M36}\cite{Z59}.  Because of this, modern papers on the subject often refer to this class of problems as Marcinkiewics-type discretization problems. 
It has been recently proven that for all $1\leq p<\infty$ there are certain entropy conditions on $X^N\subseteq L_p(\Omega)$ which guarantee that the $L_p$-norm on $X^N$ can be discretized to satisfy \eqref{E:discI} using $M$ on the order of $N(\log(N))^2$ for $1\leq p\leq2$ and  $N(\log(N))^p$ for $p\geq2$ sampling points \cite{DPTT19},\cite{DPSTT21},\cite{K21},\cite{T18}. 
These entropy conditions can be fairly technical, but they imply in particular that there exists $\beta>0$ such that
\begin{equation}\label{E:Nik}
\|x\|_{L_\infty} \leq \beta N^{1/p}  \|x\|_{L_p}\hspace{1cm}\textrm{ for all }x\in X^N.
\end{equation}
The above inequality is referred to as a Nikol’skii-type inequality or $(p,\infty)$-Nikol'skii-type inequality.
Inequality \ref{E:Nik} is a very natural Banach space condition as it is equivalent to for all $t\in\Omega$ point evaluation at $t$ is a bounded linear functional on $X^N$ with norm at most $\beta N^{1/p}$.  For the case $p=2$, the celebrated solution to the Kadison-Singer problem \cite{MSS15} can be applied to show that if $X^N\subseteq L_2(\Omega)$ satisfies the inequality \eqref{E:Nik} then the $L_2$-norm on $X^N$ can be discretized using $M$ on the order of $N$ sampling points \cite{LT21}.  The main theorem in \cite{MSS15} was first used in discretization to prove that if $\Omega\subseteq\R$ is a subset with finite measure then $L_2(\Omega)$ has a frame of exponentials \cite{NOU16}, and was later used to prove that if $(x_t)_{t\in\Omega}$ is a bounded continuous frame of a Hilbert space $H$ then there is a sampling $(t_j)_{j\in J}\subseteq \Omega$ such that $(x_{t_j})_{j\in J}$ is a frame of $H$ \cite{FS19}.  

It was not previously known that if $X_N\subseteq L_p(\Omega)$ satisfied the Nikol'skii-type inequality \ref{E:Nik} for some $1\leq p<\infty$ with $p\neq 2$ then  the $L_p$-norm on $X^N$ could be discretized using $M$ on the order of $N$ sampling points.
In Section \ref{S:L1} we give a construction of a class of subspaces $X^N\subseteq L_1(\Omega)$ which satisfy \eqref{E:Nik} but the $L_1$-norm on $X^N$ cannot be discretized using $M$ on the order of $N$ sampling points, and in Section \ref{S:Lp} we give for each $1\leq p<2$ a construction of a class of subspaces $X^N\subseteq L_1(\Omega)$ which satisfy \eqref{E:Nik} but the $L_p$-norm on $X^N$ cannot be discretized using $M$ on the order of $N$ sampling points.  The case $2<p<\infty$ is still open.

In Section \ref{S:phase} we introduce the topics of frame theory and phase retrieval, and then show the important connection between them and norm discretization of subspaces of $L_p$.  In particular, we prove that discretizing a continuous frame to obtain a discrete frame which does stable phase retrieval requires discretizing both the $L_2$ norm and the $L_1$ norm on the range of the analysis operator of the continuous frame.

We explain in Section \ref{S:phase} that the solution to the discretization problem for continuous frames \cite{FS19} implies that there exists constants $A,B>0$ such that if $\Omega$ is any $\sigma$-finite measure space and $X\subseteq L_2(\Omega)$ is a subspace such that there exists $\beta>0$ with $\|x\|_{L_\infty}\leq \beta\|x\|_{L_2}$ for all $x\in X$ then there exists sampling points $(t_j)_{j\in J}\subseteq \Omega$ such that $\beta^2 A\|x\|_{L_2}\leq \sum_{j\in J}|x(t_j)|^2\leq \beta^2 B\|x\|_{L_2}$ for all $x\in X$.
In Section \ref{S:infinite} we use the finite dimensional results in Sections \ref{S:L1} and \ref{S:Lp} to prove that the corresponding result strongly fails for $1\leq p<2$.  That is, for all $1\leq p<2$ there exists a subspace $X\subseteq L_p(\R)$ with $X$ isomorphic to $\ell_p$ and $\|x\|_{L_\infty}\leq \|x\|_{L_p}$ for all $x\in X$ such that if $(t_j)_{j\in J}\subseteq \R$ is such that $\sum_{j\in J}|x(t_j)|^p>0$ for all $x\in X\setminus\{0\}$ then there exists $y\in X$ with  $\sum_{j\in J}|y(t_j)|^p=\infty$.

\section{Constructing subspaces of $L_1$}\label{S:L1}
In this section we will show how to construct subspaces of $L_1$ where the $L_1$-norm cannot be discretized using a number of sampling points on the order of the dimension of the subspace.  In particular, 
for all $\vp>0$, we construct a class of subspaces $X^N\subseteq L_1[0,1]$ of the form $X_N=T_N(\spn (\mathbbm{1}_{[\nicefrac{(j-1)}{N},\nicefrac{j}{N})})_{j=1}^N)$ where $T_N:L_1[0,1]\rightarrow L_1[0,1]$ is a linear operator such that $\|I_{L_1[0,1]}-T_N\|<\vp$ (where $I_{L_1[0,1]}$ is the identity operator on $L_1[0,1]$).  The subspace $X^N\subseteq L_1[0,1]$ that we construct  satisfies $\|f\|_{L_\infty}\leq (1+\vp)N\|f\|_{L_1}$ for all $f\in X^N$ and yet  the $L_1$-norm on $X^N$ cannot be discretized using $M=o(N\log(N)/(\log\log(N))$ sampling points.  That is, we can consider the simplest $N$-dimensional subspace of $L_1[0,1]$ and perturb it an arbitrarily small amount to create a subspace which still satisfies the boundedness condition \eqref{E:Nik} and yet the subspace cannot be discretized using $M=o(N\log(N)/(\log\log(N))$ sampling points.

We now describe how to construct the subspace $X^N\subseteq L_1[0,1]$.  Let $\vp>0$ and $n\in\N$.  Without loss of generality, we may assume that $1/\vp$ is an integer and let $N=n+(\vp^{-1}n)^n$.  We will construct a basis $(x_j)_{j=1}^{N}$ of $X_N$ which will be a perturbation of the sequence of indicator functions $(\mathbbm{1}_{[\nicefrac{(j-1)}{N}, \nicefrac{j}{N})})_{j=1}^N$.  For $1 \leq j \leq n$ we let,
$$x_j=N \mathbbm{1}_{[\nicefrac{(j-1)}{N}, \nicefrac{j}{N})}.$$

%--------Plots-------------

%2
\setlength{\unitlength}{.02in}%{.025in}
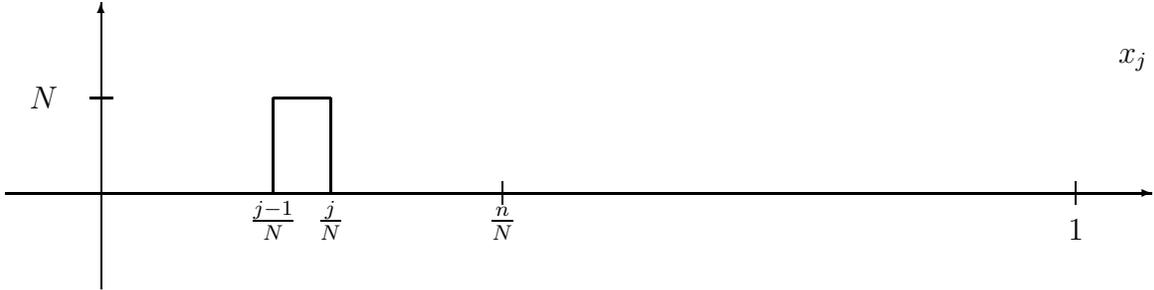
\begin{picture}(250,75)
\put(0,25){\vector(1,0){300}}
\put(25,0){\vector(0,1){75}}
\put(280,22){\line(0,1){6}}
\put(22,50){\line(1,0){6}}
\put(130,22){\line(0,1){6}}
\thicklines
\put(70,50){\line(1,0){15}}
\put(85,25){\line(0,1){25}}
\put(70,25){\line(0,1){25}}
\put(10,50){\makebox(0,0){$\small{N}$}}
\put(70,13){\makebox(0,0)[b]{$\tiny \frac{j-1}{N}$}}
\put(280,13){\makebox(0,0)[b]{$\small{1}$}}
\put(130,13){\makebox(0,0)[b]{$\tiny \frac{n}{N}$}}
\put(85,13){\makebox(0,0)[b]{$\tiny \frac{j}{N}$}}

\put(295,60){\makebox(0,0){$\small{x_j}$}}
\end{picture}

For each $1\leq k\leq n$ we partition the interval $[\nicefrac{(k-1)}{N},\nicefrac{k}{N})$ into $\vp^{-1}n$ intervals $(I_k(i))_{i=1}^{\vp^{-1}n}$ each of width $\vp n^{-1} N^{-1}$.  As $N=n+(\vp^{-1}n)^n$, we can now enumerate $\{1,2,...,\vp^{-1}n\}^n$ as $((i_{k,j})_{k=1}^n)_{j=n+1}^N$.  For $n < j \leq N$ we let,
$$x_j= N \sum_{k=1}^n \mathbbm{1}_{I_k(i_{k,j})}+N \mathbbm{1}_{[\frac{j-1}{N}, \frac{j}{N})}. $$
\vspace{0.2in}

%4
\setlength{\unitlength}{.02in}%{.025in}
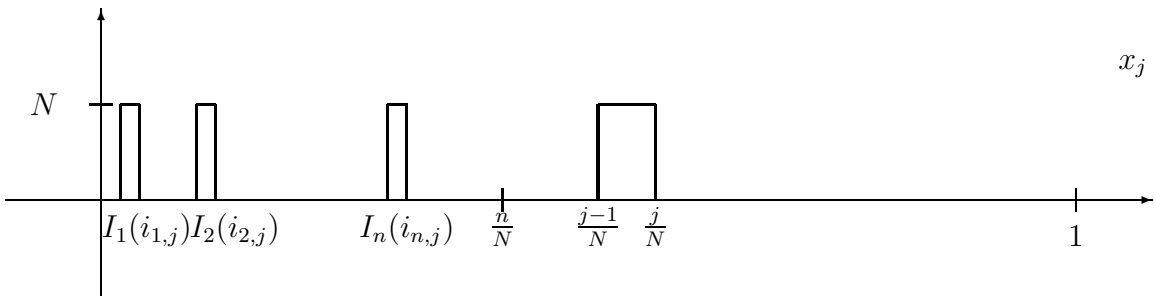
\begin{picture}(250,75)
\put(0,25){\vector(1,0){300}}
\put(25,0){\vector(0,1){75}}
%\put(32,22){\line(0,1){6}}
\put(130,22){\line(0,1){6}}
\put(22,50){\line(1,0){6}}
\put(280,22){\line(0,1){6}}
\thicklines
\put(30,50){\line(1,0){5}}
\put(35,25){\line(0,1){25}}
\put(30,25){\line(0,1){25}}

\put(50,50){\line(1,0){5}}
\put(55,25){\line(0,1){25}}
\put(50,25){\line(0,1){25}}

\put(100,50){\line(1,0){5}}
\put(105,25){\line(0,1){25}}
\put(100,25){\line(0,1){25}}

\put(155,50){\line(1,0){15}}
\put(170,25){\line(0,1){25}}
\put(155,25){\line(0,1){25}}

\put(10,50){\makebox(0,0){$\small{N}$}}
\put(37,13){\makebox(0,0)[b]{$\tiny {I_1} (i_{1,j})$}}
\put(60,13){\makebox(0,0)[b]{$\small{I_2} (i_{2,j})$}}
\put(105,13){\makebox(0,0)[b]{$\tiny {I_n} (i_{n,j})$}}
\put(280,13){\makebox(0,0)[b]{$\small{1}$}}
\put(295,60){\makebox(0,0){$\small{x_j}$}}
\put(170,13){\makebox(0,0)[b]{$\tiny \frac{j}{N}$}}
\put(155,13){\makebox(0,0)[b]{$\tiny \frac{j-1}{N}$}}
\put(130,13){\makebox(0,0)[b]{$\tiny \frac{n}{N}$}}
\end{picture}

\vspace{0.2in}

We have created a sequence $(x_j)_{j=1}^N\subseteq L_1[0,1]$, and we now prove that it satisfies the following theorem.

\begin{thm}\label{T:L1}
Let $n\in\N$ and $1>\vp>0$ with $\vp^{-1}\in\N$.  Consider $N=n+(\vp^{-1}n)^n$ and $X^N\subseteq L_1[0,1]$ with basis $(x_j)_{j=1}^N$ as defined above.  Then the following holds.

\begin{enumerate}
    \item There is a linear map $T_N:L_1[0,1]\rightarrow L_1[0,1]$ with $T_N(\mathbbm{1}_{[\nicefrac{(j-1)}{N},\nicefrac{j}{N})})=x_j$ for all $1\leq j\leq N$ such that $\|I_{L_1[0,1]}-T_N\|\leq \vp$, where $I_{L_1[0,1]}$ is the identity operator.
    \item The basis $(x_j)_{j=1}^N$ of $X^N$ satisfies that
    $$(1-\vp)\sum_{j=1}^N|a_j|\leq \Big\|\sum_{j=1}^N a_j x_j\Big\|_{L_1}\leq (1+\vp) \sum_{j=1}^N |a_j|\quad\textrm{ for all }(a_j)_{j=1}^N\in\ell_1^N,
    $$
    \item For all $x\in X^N$, $\|x\|_{L_\infty[0,1]}\leq (1-\vp)^{-1} N \|x\|_{L_1[0,1]}$.
    \item Let $M\in\N$ so that $(t_j)_{j=1}^M\subseteq [0,1]$ and  $\sum |x(t_j)|\neq 0$ for all $x\in X^N\setminus\{0\}$.
        Then there exists $x\in X^N$ with
        $$\frac{nN}{(1+\vp)M} \|x\|_{L_1} \leq \frac{1}{M} \sum_{j=1}^M |x(t_j)| $$
\end{enumerate}
In particular, it follows from (4) that the $L_1$ norm on $X^N$ cannot be discretized using $M$ on the order of $N$ sampling points.  Furthermore, the $L_1$ norm on $X^N$ cannot be discretized using $M=o(N\log(N)/(\log\log(N))$ sampling points.  However,  $M$ can be chosen on the order of $N\log(N)/(\log\log(N))$ so that there exists sampling points $(t_j)_{j=1}^M\subseteq [0,1]$ such that $\|x\|_{L_1} = \frac{1}{M} \sum_{j=1}^M |x(t_j)|$ for all $x\in X^N$.
\end{thm}

\begin{proof}
Consider the linear map $P_N:L_1[0,1]\rightarrow L_1[0,1]$  given by $P_N=\sum_{j=1}^N \E_{[\nicefrac{(j-1)}{N},\nicefrac{j}{N})}$ where $\E_{[\nicefrac{(j-1)}{N},\nicefrac{j}{N})}$ is conditional expectation on $[\nicefrac{(j-1)}{N},\nicefrac{j}{N})$.  Note that $P_N$ is a projection onto the subspace $\spn (\mathbbm{1}_{[\nicefrac{(j-1)}{N},\nicefrac{j}{N})})_{j=1}^N$ with $\|P_N\|=1$.  We define the operator $S_N:\spn (\mathbbm{1}_{[\nicefrac{(j-1)}{N},\nicefrac{j}{N})})_{j=1}^N\rightarrow L_1[0,1]$ by $S_N(\sum_{j=1}^N a_j N \mathbbm{1}_{[\nicefrac{(j-1)}{N},\nicefrac{j}{N})})=\sum_{j=1}^N a_j x_j$. We now let $T_N=(I_{L_1[0,1]}-P_N)+S_N P_N$.  We have the following estimate.
\begin{align*}
    \|I_{L_1[0,1]}-T_N\|&=\|S_N P_N-P_N\|\\
    &=\sup_{\sum a_j=1}\Big\|\sum_{j=1}^N a_j x_j-\sum_{j=1}^N a_j N\mathbbm{1}_{[\nicefrac{(j-1)}{N},\nicefrac{j}{N})}\Big\|_{L_1}\\
    &\leq\sup_{\sum a_j=1}\sum_{j=n+1}^N \big\|a_j x_j-a_j N \mathbbm{1}_{[\nicefrac{(j-1)}{N},\nicefrac{j}{N})}\big\|_{L_1}\\
&=\sup_{\sum a_j=1}\sum_{j=n+1}^N a_j N\Big\|\sum_{k=1}^n \mathbbm{1}_{I_k(i_{k,j})}\Big\|_{L_1}\\
&=\vp \hspace{1cm}\textrm{ \small (because $I_k(i_{k,j})$ has length $\vp n^{-1}N^{-1}$ for all $1\leq k\leq n$).}\\
\end{align*}
This shows that $\|I_{L_1[0,1]}-T_N\|\leq \vp$ which proves (1).  We have that $\|x_j\|_{L_1}\leq 1+\vp$ for all $1\leq j\leq N$ and hence $\|\sum_{j=1}^N a_j x_j\|_{L_1}\leq (1+\vp)\sum_{j=1}^N |a_j|$ for all $(a_j)_{j=1}^N\in\ell_1^N$.  We now consider a fixed $(a_j)_{j=1}^N\in\ell_1^N$ with $\sum_{j=1}^N|a_j|=1$.
\begin{align*}
    \Big\|\sum_{j=1}^Na_j x_j\Big\|_{L_1}&=\Big\|\sum_{j=1}^N a_j N\mathbbm{1}_{[\nicefrac{(j-1)}{N},\nicefrac{j}{N})}-a_j(N\mathbbm{1}_{[\nicefrac{(j-1)}{N},\nicefrac{j}{N})}-x_j)\Big\|_{L_1}\\
    &\geq\sum_{j=1}^N |a_j| -\sum_{j=1}^N|a_j| \big\|(I_{L_1[0,1]}-T_N)N\mathbbm{1}_{[\nicefrac{(j-1)}{N},\nicefrac{j}{N})}\big\|_{L_1}\\
    &\geq \sum_{j=1}^N |a_j| -\|I_{L_1[0,1]}-T_N\|\sum_{j=1}^N|a_j|\\
    &=1-\vp
\end{align*}
Thus, we have proven (2) by showing that $$(1-\vp)\sum_{j=1}^N|a_j|\leq \Big\|\sum_{j=1}^N a_j x_j\Big\|_{L_1}\leq (1+\vp) \sum_{j=1}^N |a_j| \,\,\,\text{for all}\,\,\, (a_j)_{j=1}^N\in\ell_1^N.$$

Note that $\|x_j\|_{L_\infty}=N$ for all $1\leq j\leq N$.  Thus for all $(a_j)_{j=1}^N\in\ell_1^N$ we have by (2) that
$$\Big\|\sum_{j=1}^N a_j x_j\Big\|_{L_\infty}\leq \sum_{j=1}^N|a_j|N\leq(1-\vp)^{-1} N \Big\|\sum_{j=1}^N a_j x_j\Big\|_{L_1}.
$$
Thus we have proven (3).

We now let $M\in\N$ and $(t_k)_{k=1}^M\subseteq[0,1]$ with $\sum |x(t_k)|\neq 0$ for all $x\in X^N$.  Hence, after reordering $(t_k)_{k=1}^M$ we may assume without loss of generality that $|x_k(t_k)|=1$ for all $1\leq k\leq n$.  There exists unique $n< j\leq N$ such that $t_k\in I_k(i_{j,k})$ for all $1\leq k\leq n$.  Thus, we have that
$$
    \sum_{k=1}^M |x_j(t_k)|\geq \sum_{k=1}^n N \mathbbm{1}_{I_k(i_{k,j})}(t_k)
    = nN.
$$
As $\|x_j\|_{L_1}\leq 1+\vp$, this proves (4).  To prove that the $L_1$ norm on $X^N$ cannot be discretized using $M=o(N\log(N)/(\log\log(N))$ sampling points it follows from (4) that we need to prove that there exists $C>0$ and $n_0\in\N$ such that $N\log(N)/(\log\log(N)\leq CNn$ for all $n\geq n_0$.  We have that $N=n+(\vp^{-1}n)^n$.  Thus if $n\in\N$ is large enough then $N\leq n^{2n}$ and hence $\log(N)\leq 2n\log(n)$. On the other hand, $N\geq (\vp^{-1}n)^n$ and hence $\log(N)\geq n$ which implies that $\log\log(N)\geq \log(n)$.  Thus we have if $n\in\N$ is large enough then
$$2Nn\geq \dfrac{N \log(N)}{\log(n)}\geq \dfrac{N\log(N)}{\log\log(N)}.
$$

We now prove that the $L_1$ norm on $X^N$ can be perfectly discretized using $M$ on the order of $N\log(N)/(\log\log(N)$ sampling points.  For $n\in\N$ we let $M=\vp^{-1}nN$ and $t_j=\frac{j-1}{M}$.  As each $x\in X^N$ is constant on the interval $[\frac{j-1}{M}, \frac{j}{M})$ for all $1\leq j\leq M$, we have that $\|x\|_{L_1} = \frac{1}{M} \sum_{j=1}^M |x(t_j)|$ for all $x\in X^N$.  Thus we just need to prove that there exists $C>0$ and $n_0\in\N$ such that $CN\log(N)/\log\log(N)\geq Nn$ for all $n\geq n_0$.  

As $N=n+(\vp^{-1}n)^n$ we have that $N\geq n^n$ and hence $\log(N)\geq n\log(n)$.  On the other hand, if $n\in\N$ is large enough then $\log(N)\leq n^{2}$ and hence $\log\log(N)\leq 2\log(n)$.  Thus we have if $n\in\N$ is large enough then
$$Nn\leq \dfrac{N \log(N)}{\log(n)}\leq \dfrac{2N\log(N)}{\log\log(N)}.
$$
\end{proof}

\section{Constructing subspaces of $L_p$ for $1\leq p<2$ }\label{S:Lp}

In section \ref{S:L1} we constructed subspaces $X^N \subseteq L_1[0,1]$ which satisfied the Nikol’skii-type inequality $\|x\|_{L_\infty}\leq (1+\vp)\|x\|_{L_1}$ for all $x\in X^N$ and yet the $L_1$ norm on $X^N$ cannot be discretized using $M=o(N\log(N)/\log\log(N))$ sampling points.  That construction only works for $L_1$, but we now introduce a different method which works for $L_p$ when $1\leq p<2$.   

For $n\in\N$, the $n$th Rademacher function $R_n$ on $[0,1]$ is given by
$$R_n=\sum_{j=1}^{2^n} (-1)^j \mathbbm{1}_{[\nicefrac{(j-1)}{2^n},\nicefrac{j}{2^n})}.$$  Note that $(R_n)_{n=1}^\infty$ is an independent sequence of mean-zero $\pm1$ random variables on $[0,1]$.  This sequence is very useful when using probabilistic techniques in the geometry of Banach spaces, and we will rely on the following theorem.  

\begin{thm}[Khintchine's Inequality] For all $1\leq p<\infty$ 
there exists constants $0<A_p\leq B_p$ such that if $(R_j)_{j=1}^N$ is a sequence of Rademacher functions on $[0,1]$ then for all $N\in\N$ and all scalars $(a_j)_{j=1}^N$, we have that
$$
A_p \Big(\sum_{j=1}^N |a_j|^2\Big)^{1/2}\leq \Big(\int \big|\sum_{j=1}^N a_j R_j(s)\big|^p  ds \Big)^{1/p}\leq B_p \Big(\sum_{j=1}^N |a_j|^2\Big)^{1/2}
$$
\end{thm}

 We now use the Rademacher functions and Khintchine's inequality to build a class of subspaces of $L_p[0,1]$ for $1\leq p<2$ where the $L_p$ norm cannot be discretized using $M$ on the order of the dimension number of sampling points.

\begin{prop}\label{P:bad}
Let $1\leq p<2$ and $N\in\N$. For each $1\leq j\leq N$ let $y_j(t)=N^{1/p-1/2} R_j(N^{1-p/2 }t)$ and $X^N=\spn(y_j)_{j=1}^N$.  Then,
\begin{enumerate}
    \item $(y_j)_{j=1}^N$ is $1$-equivalent to the Rademacher sequence $(R_j)_{j=1}^N$ in $L_p$.
    \item $\|x\|_{L_\infty}\leq A^{-1}_p N^{1/p}\|x\|_{L_p}$ for all $x\in X^N$ where $A_p$ is the constant in Khintchine's inequality.
    \item Suppose that $(t_j)_{j=1}^M\subseteq[0,1]$ are such that $\sum_{j=1}^M |x(t_j)|^p>0$ for all $x\in X^N\setminus\{0\}$. Then $$ \frac{N^{2-p/2}}{M}\|y_1\|_{L_p}^{p}\leq \frac{1}{M}\sum_{j=1}^M |y_1(t_j)|^p.$$ 
\end{enumerate}
Thus, the $L_p$-norm on $X^N$ cannot be discretized using $M$ on the order of $N$ sampling points.
\end{prop}

\begin{proof}
Let $(a_j)_{j=1}^N\in\ell_2^N$.  We have that
\begin{align*}
    \Big\|\sum_{j=1}^N a_j y_j\Big\|^p_{L_p}&=\int \big|\sum_{j=1}^N a_j N^{1/p-1/2} R_j(N^{1-p/2} t)\big|^p dt\\
    &=\int \big|\sum_{j=1}^N a_j R_j(s)\big|^p ds \quad\textrm{ \small by substituting }s=N^{1-p/2}t.
\end{align*}
Thus, $(y_j)_{j=1}^N$ is $1$-equivalent to the Rademacher sequence $(R_j)_{j=1}^N$ in $L_p[0,1]$. By Khintchine's inequality we have that
\begin{equation}\label{E:KI}
A_p \Big(\sum_{j=1}^N |a_j|^2\Big)^{1/2}  \leq  \Big\|\sum_{j=1}^N a_j y_j\Big\|_{L_p}\leq B_p \Big(\sum_{j=1}^N |a_j|^2\Big)^{1/2}.
\end{equation}
 Let $x=\sum_{j=1}^N a_j y_j\in X^N$.  We now give an upper bound for $\|x\|_{L_\infty}$. 
\begin{align*}
    \|x\|_{L_\infty}&=\sup_{t\in[0,1]}\Big|\sum_{j=1}^N a_j N^{1/p-1/2} R_j(N^{1-p/2}t)\Big|\\
    &=N^{1/p}\sum_{j=1}^N N^{-1/2}|a_j|\\
    &\leq N^{1/p}\Big(\sum_{j=1}^N|a_j|^2\Big)^{1/2}\quad\textrm{ \small by Cauchy-Schwartz,}\\
    &\leq A_p^{-1} N^{1/p}\|x\|_{L_p}\quad \quad \quad \textrm{ by \eqref{E:KI}.}
\end{align*}
Thus we have proven (2).  We now suppose that $(t_j)_{j=1}^M\subseteq[0,1]$ are such that $\sum_{j=1}^M |x(t_j)|^p>0$ for all $x\in X^N$.  As $X^N$ is $N$-dimensional and supported on $[0,N^{p/2-1}]$ we have that there exists a subset $(t_j)_{j\in J}\subseteq[0,N^{p/2-1}]$ with $|J|\geq N$.  Note that, $|y_1(t_j)|=N^{1/p-1/2}$ for all $j\in J$. We thus have that,
$$
    \sum_{j=1}^M |y_1(t_j)|^p \geq |J| N^{1-p/2}\geq N N^{1-p/2} = N^{2-p/2}
$$
As $\|y_1\|_{L_p}=1$ we have proven (3).
\end{proof}

In Proposition \ref{P:bad} we  constructed for all $1\leq p<2$ a collection of subspaces $(X^N)_{N=1}^\infty$ of $L_p[0,1]$ where $X^N$ is $(B_p/A_p)$-isomorphic to $\ell_2^N$ and satisfies $\|x\|_{L_\infty}\leq A_p^{-1}\|x\|_{L_p}$ for all $x\in X^N$, and yet the $L_p$-norm on $X^N$ could not be discretized using $M$ on the order of $N$ sampling points.  Our proof only worked for $1\leq p<2$, but it is natural to ask if the result was still true for $2<p<\infty$.  We now show that this hypothesis fails for all $2<p$.  That is, there does not exist uniform constants $C_p,D_p>0$ such that for all $N\in\N$, there is a probability space $(M,\Sigma,\mu)$ such that $\ell_2^N$ is $C_p$-isomorphic to a subspace $X^N\subseteq L_p(\mu)$ with $\|x\|_{L_\infty(\mu)}\leq D_pN^{1/p}\|x\|_{L_\infty(\mu)}$ for all $x\in X^N$.  The proof follows the classical argument in \cite{FLP77} which proves that for all $2<p$ there does not exist a constant $C>0$ so that $\ell_2^n$ uniformly embeds into $\ell_p^{Cn}$.

\begin{prop}\label{P:p>2}
Let $2<p<\infty$ and $N\in\N$.   Let $X^N \subseteq L_p(\Omega)$ where $\Omega$ is a probability space and $X^N$ has a basis $(x_j)_{j=1}^N$ such that $\|x_j\|_{L_p(\Omega)}=1$ and the following holds for some $A,B,\beta>0$:
\begin{enumerate}
    \item  For all $(a_j)_{j=1}^N\in\ell_2^N$,
    $$A \, \Big(\sum_{j=1}^N |a_j|^2\Big)^{1/2} \leq \Big\|\sum_{j=1}^N a_j x_j\Big\|_{L_p(\Omega)}\leq B\Big(\sum_{j=1}^N |a_j|^2\Big)^{1/2} $$
    \item For all $x \in X^N$,
    $$ \|x\|_{L_{\infty}(\Omega)} \leq  \beta N^{1/p} \|x\|_{L_p{(\Omega)}}$$
\end{enumerate}
Then, $\frac{A}{B \, B_p} N^{1/2-1/p} \leq \beta$ where $B_p$ is the constant from Khintchine's inequality.  In particular, there does not exist constants $A,B,\beta>0$ such that for all $N\in\N$ there exists a subspace $X^N\subseteq L_p(\Omega)$ such that $X^N$ is $A^{-1}B$-isomorphic to $\ell_2^N$ and $X^N\subseteq L_p(\Omega)$ satisfies the boundedness condition $ \|x\|_{L_{\infty}(\Omega)} \leq  \beta N^{1/p} \|x\|_{L_p{(\Omega)}}$ for all $x\in X^N$.
\end{prop}
 
\vspace{0.2in}
\begin{proof}
Let $t \in \Omega$. By (2) we have without loss of generality that
\begin{equation} \label{E:H1}
|x(t)| \leq \beta N ^{1/p}\|x\|_{L_p(\Omega)}\hspace{1cm}\textrm{ for all }x\in X^N.
\end{equation}
 By scaling (1), we have the following inequality
\[A \leq \Bigg\|\sum_{j=1}^N \dfrac{x_j(t)}{\left( \sum_{j=1}^N |x_j(t)|^2\right)^{1/2}} \, x_j\Bigg\|_{L_p(\Omega)}\leq B\]
Now using inequality (\ref{E:H1}) for $x=\sum_{j=1}^N \frac{x_j(t)}{\left( \sum_{j=1}^N |x_j(t)|^2\right)^{1/2}}  x_j$ we get that
\[\sum_{j=1}^N \dfrac{|x_j(t)|^2}{\left( \sum_{j=1}^N |x_j(t)|^2\right)^{1/2}} \leq \beta N^{1/p} B\]
Therefore,
\begin{equation}\label{E:H2}
\left( \sum_{j=1}^N |x_j(t)|^2\right)^{1/2}  \leq \beta N^{1/p} B    
\end{equation}

Let $(R_j)_{j=1}^\infty$ be the sequence of Rademacher functions on $[0,1]$.  For $s \in [0,1]$ we have by (1) that
$$
  A^{p} N^{p/2} \leq \Big\| \sum_{j=1}^N R_{j}(s) x_j\Big\|_{L_p(\Omega)}^p = \int_{\Omega} \Big|\sum_{j=1}^N R_{j}(s) x_j(t)\Big|^p dt
$$

By integrating with respect to $s$ we get
\begin{align*}
A^{p} N^{p/2} & \leq \int_{0}^{1} \int_{\Omega} \Big|\sum_{j=1}^N R_{j}(s) x_j(t)\Big|^p dt \, 
ds\\
&=  \int_{\Omega}\int_{0}^{1} \Big|\sum_{j=1}^N R_{j}(s) x_j(t)\Big|^p ds \, 
dt\\
& \leq \int_{\Omega} \left( B_p \left( \sum_{j=1}^N |x_j(t)|^2\right)^{1/2}\right)^p dt \hspace{1cm}\text{\small (by  Khintchine's inequality)}\\
&\leq \int_{\Omega} \left( B_p \beta N^{1/p} B\right)^p dt\hspace{3cm}  \text{by \eqref{E:H2}}\\
&= \left( B_p \beta N^{1/p} B\right)^p
\end{align*}

Thus we have that $A N^{1/2} \leq B_p \beta N^{1/p} B$ and hence $\frac{A}{B \, B_p} N^{1/2-1/p} \leq \beta$.  As $2<p$, we have that $N^{1/2-1/p}$ is unbounded and hence a uniform constant $\beta$ cannot exist.
\end{proof}

\section{Discretization and frame theory}\label{S:phase}

In previous sections we have considered the problem of discretizing the $L_p$ norm on a subspace $X\subseteq L_p$.  We now show how this problem is naturally connected with frame theory.  
A family of vectors $(x_j)_{j \in J}$ in a Hilbert space $H$ is called a {\em frame} of $H$ if there are constants $0<A\leq B<\infty$ so that for all $x\in \HH^N$,
\begin{equation}\label{E:frame}
A \|x\|^2 \leq \sum_{j \in J} |\langle x, x_j\rangle|^2\leq B \|x\|^2.
\end{equation}
A frame is called {\em tight} if the optimal frame bounds satisfy $A=B$, and a frame is called {\em Parseval} if the optimal frame bounds satisfy $A=B=1$.  The {\em analysis operator} of a frame $(x_j)_{j \in J}$ of $H$ is the map $T:H\rightarrow \ell_2(J)$ given by $T(x)=(\langle x,x_j\rangle)_{j\in J}$.   Note that $(x_j)_{j \in J}$ has upper frame bound $B$ and lower frame bound $A$ if and only if the analysis operator is an embedding and satisfies $A\|x\|^2\leq \|Tx\|^2\leq B\|x\|^2$ for all $x\in H$.

The notion of a frame can be generalized to a continuous frame by changing the summation in \eqref{E:frame} to integration over a measure space. A  collection of vectors $(x_t)_{t\in \Omega}$ in a Hilbert space $H$ is called a {\em continuous frame} of $H$ over a measure space $(\Omega,\Sigma,\mu)$ if there are constants $0<A\leq B<\infty$ so that for all $x\in H$,
\begin{equation}\label{E:frame2}
A \|x\|^2 \leq \int_\Omega |\langle x, x_j\rangle|^2 d\mu\leq B \|x\|^2.
\end{equation}
A continuous frame is called {\em tight} if the optimal frame bounds satisfy $A=B$, and a continuous frame is called {\em Parseval} if the optimal frame bounds satisfy $A=B=1$. The {\em analysis operator} of a continuous frame $(x_t)_{t \in \Omega}$ of $H$ is the map $T:H\rightarrow L_2(\Omega)$ given by $T(x)=(\langle x,x_t\rangle)_{t\in \Omega}$.  Note that the analysis operator of a frame $(x_j)_{j\in J}$ is an embedding of $H$ into $\ell_2(J)$ and that the analysis operator of a continuous frame $(x_t)_{t\in \Omega}$ is an embedding of $H$ into $L_2(\Omega)$.

Continuous frames are widely used in mathematical physics and are particularly prominent in quantum mechanics and quantum optics. Though
 continuous frames such as the short time Fourier transform naturally  characterize many different physical properties,  discrete frames are much better suited for  computations.  Because of this, when working with continuous frames, researchers often create a discrete frame by sampling the continuous frame and then use the discrete frame for computations instead of the entire continuous frame.
 That is given, a continuous frame  $(x_t)_{t \in \Omega}$ of $H$ we are interested in choosing $(t_j)_{j\in J}\subseteq \Omega$ so that $(x_{t_j})_{j\in J}$ is a frame of $H$.
The notion of creating a frame by sampling a continuous frame has its origins in the very start of modern frame theory.
Indeed, Daubechies, Grossmann, and Meyer \cite{DGM86} popularized modern frame theory in their seminal paper ``Painless nonorthogonal expansions", and their constructions of frames for Hilbert spaces were all done by sampling different continuous frames. Sampling continuous frames continues to be an important subject in applied harmonic analysis, and there are many modern research papers on the subject in various contexts \cite{AHP19}\cite{DMM21}\cite{DKNRV21}.
The discretization problem, posed by Ali, Antoine, and Gazeau in
their physics textbook {\em Coherent States, Wavelets, and Their
Generalizations} \cite{AAG00},  asks when a continuous frame of a
Hilbert space can be sampled to obtain a frame.  They state that a
positive answer to the question is crucial for practical
applications of coherent states, and chapter 16 of the book is
devoted to the discretization problem. The first author
with Darrin Speegle solved the discretization problem in its full
generality by characterizing exactly when a continuous frame may be
sampled to obtain a frame \cite{FS19}.  

\begin{thm}[\cite{FS19}]
Let $(x_t)_{t\in \Omega}$ be a continuous frame of a separable Hilbert space $H$ over a measure space $(\Omega,\Sigma,\mu)$ such that singletons are measurable.  Then there exists $(t_j)_{j\in J}\subseteq \Omega$ such that $(x_{t_j})_{j\in J}$ is a frame of a $H$ if and only if there is a measure $\nu$ on $(\Omega,\Sigma)$ such that $(x_t)_{t\in \Omega}$ is a continuous frame of a $H$ over the measure space $(\Omega,\Sigma,\nu)$ and there is a constant $\beta>0$ so that $\|x_t\|\leq \beta$ for almost every $t\in\Omega$.
\end{thm}

Furthermore, they prove the following quantized version.

\begin{thm}[\cite{FS19}]
There exists uniform constants $A,B>0$ such that the following holds.
Let $(x_t)_{t\in \Omega}$ be a continuous Parseval frame of a separable Hilbert space $H$ over a measure space $(\Omega,\Sigma,\mu)$ such that $\|x_t\|\leq 1$ for all $t\in\Omega$.  Then there exists $(t_j)_{j\in J}\subseteq \Omega$ such that $(x_{t_j})_{j\in J}$ is a frame of $H$ and
$$A\|x\|^2\leq \sum_{j\in J}|\langle x,x_{t_j}\rangle|^2\leq B\|x\|^2\hspace{1cm}\textrm{ for all }x\in H.
$$
\end{thm}

We now consider how these results are connected to the problem of discretizing $L_p$ norms for subspaces.  If $(x_t)_{t\in \Omega}$ is a continuous Parseval frame for a Hilbert space $H$ then the analysis operator $T:H\rightarrow L_2(\Omega)$ is an isometric embedding of $H$ into $L_2(\Omega)$.  Being able to discretize the continuous frame $(x_t)_{t\in \Omega}$ to get a frame  $(x_{t_j})_{j\in J}$ of $H$ with lower frame bound $A$ and upper frame bound $B$ is then equivalent to discretizing the $L_2$ norm on the range of the analysis operator $T(H)\subseteq L_2(\Omega)$ so that $$A\|y\|^2\leq \sum_{j\in J}|y(t_j)|^2\leq B\|y\|^2 \,\,\,\text{for all} \,\,\, y \in T(H)$$ where if $y=Tx$ then $y(t)=\langle x, x_t\rangle$ for all $t\in \Omega$.  Furthermore, suppose that $Y\subseteq L_2(\Omega)$ is a closed subspace and $\beta>0$.  Then, $Y\subseteq L_2(\Omega)$ satisfies that $\|y\|_{L_\infty}\leq \beta\|y\|_{L_2}$ for all $y\in Y$ if and only if $Y$ is the range of the analysis operator of a continuous Parseval frame $(x_t)_{t\in\Omega}$ of a Hilbert space $H$ such that $\|x_t\|\leq \beta$ for all $t\in\Omega$. This gives the following corollary.

\begin{cor}\label{C:disc}
There exists uniform constants $A,B>0$ such that the following holds.
Let $(\Omega,\mu)$ be a $\sigma$-finite measure space and $\beta\geq 1$.  Suppose that $Y\subseteq L_2(\Omega)$ is a closed subspace such that $\|y\|_{L_\infty}\leq \beta\|y\|_{L_2}$ for all $y\in Y$.  Then there exists $(t_j)_{j\in J}\subseteq \Omega$ such that 
$$\beta^2 A\|y\|^2\leq \sum_{j\in J}|y(t_j)|^2\leq \beta^2 B\|y\|^2\hspace{1cm}\textrm{ for all }y\in Y.
$$
\end{cor}

Note that Corollary \ref{C:disc} applies to $\Omega$ being either a finite or infinite measure space and to $Y\subseteq L_2(\Omega)$ being either finite or infinite dimensional.  
In the case that $\Omega$ is a probability space then the following theorem gives the relationship between the dimension of the subspace, the $L_\infty$-bound on the subspace, and a bound on the number of sampling points required for discretization.

\begin{thm}[\cite{LT21}]\label{T:LT21}
There exists uniform constants $A,B,C>0$ such that the following holds.
Let $(\Omega,\mu)$ be a probability space and $\beta\geq 1$.  Suppose that $Y\subseteq L_2(\Omega)$ is a closed subspace such that $\|y\|_{L_\infty}\leq \beta N^{1/2}\|y\|_{L_2}$ for all $y\in Y$.  Then there exists $M\leq C\beta^2 N$ and sampling points $(t_j)_{j=1}^M\subseteq \Omega$ such that 
$$A\|y\|^2\leq \frac{1}{M}\sum_{j=1}^M|y(t_j)|^2\leq \beta^2 B\|y\|^2\hspace{1cm}\textrm{ for all }y\in Y.
$$
\end{thm}

One of the many applications of frame theory is in the implementation of phase retrieval.  A frame $(x_j)_{j\in J} \subseteq H$ for a Hilbert space $H$ allows for any vector $x\in H$ to be linearly recovered from the collection of frame coefficients $(\langle x, x_j\rangle)_{j\in J}$.  
 However, there are many instances in physics and engineering where one is able to obtain only the magnitude of linear measurements such as in speech recognition \cite{BR99} and X-ray crystallography \cite{Ta03}.    Let $T:H\rightarrow \ell_2(J)$ be the analysis operator of $(x_j)_{j\in J}$ given by $T(x)=(\langle x, x_j\rangle)_{j\in J}$ for all $x\in H$.   For $x\in H$, the goal of phase retrieval is to recover $x$ (up to a unimodular scalar) from the absolute value of the frame coefficients $|T(x)|=(|\langle x, x_j\rangle|)_{j\in J}$. 
 We say that a frame  $(x_j)_{j\in J}$ does phase retrieval if whenever $x,y\in H$ and $|Tx|=|Ty|$ we have that $x=\lambda y$ for some scalar $\lambda$ with $|\lambda|=1$.  We say that  $(x_j)_{j\in J}$ does {\em $C$-stable} phase retrieval if $\min_{|\lambda|=1}\|x-\lambda y\|_H\leq C\||Tx|-|Ty|\|_{\ell_2(J)}$ for all $x,y\in H$.
 If we consider the equivalence relation $\sim$ on $H$ to be $x\sim y$ if and only if $x=\lambda y$ for some $|\lambda|=1$  then a frame $(x_j)_{j\in J}$ does $C$-stable phase retrieval is equivalent to the map $|T x|\mapsto x/\!\!\sim$ is well defined and is $C$-Lipschitz.  As any application will involve some error, having a good stability bound for phase retrieval is of fundamental importance in applications.
  Likewise, if  $(x_t)_{t\in \Omega}$ is a continuous frame of $H$ with frame operator $T:H\rightarrow L_2(\Omega)$ then we say that $(x_t)_{t\in \Omega}$ does {\em $C$-stable} phase retrieval if $\min_{|\lambda|=1}\|x-\lambda y\|_H\leq C\||Tx|-|Ty|\|_{L_2(\Omega)}$ for all $x,y\in H$.

Every frame for a finite dimensional Hilbert space which does phase retrieval does $C$-stable phase retrieval for some constant $C>0$ \cite{B17}\cite{CCPW13}.  On the other hand, phase retrieval using a frame or continuous frame for an infinite dimensional Hilbert space is always unstable \cite{CCD16}\cite{AG17}.  
Given some $C>0$ and dimension $N\in\N$, it is very difficult to explicitly construct a frame of $\ell_2^N$ which does $C$-stable phase retrieval.  However, there are random constructions where it is possible to choose $C>0$ such that a frame $(x_j)_{j=1}^m$ of  random vectors does $C$-stable phase retrieval with high probability and the number of vectors $m$ can be chosen on the order of the dimension $N$ \cite{CL14}\cite{EM13}\cite{KL18}\cite{CDFF21}.  Each of these results can be thought of as sampling a continuous Parseval frame over a probability space which does stable phase retrieval to obtain a frame which does stable phase retrieval.  This naturally leads to the following problem.

\begin{prob}
Let $C,\beta>0$.  Does there exist constants $D,\kappa>0$ so that for all $N\in\N$ there exists $M\leq DN$ such that the following holds?  Suppose that $H$ is an $N$-dimensional Hilbert space, $(\Omega,\mu)$ is a probability space, and $(x_t)_{t\in \Omega}$ is a continuous  Parseval frame of $H$ which does $C$-stable phase retrieval such that $\|\psi_t\|\leq \beta\sqrt{N}$ for all $t\in \Omega$. Then there exists a sequence of sampling points $(t_j)_{j=1}^M\subseteq\Omega$ such that $(\frac{1}{\sqrt{M}}x_{t_j})_{j=1}^M$ is a frame of $H$ which does $C$-stable phase retrieval.
\end{prob}

This problem seems particularly difficult as Theorem \ref{T:LT21} which relies on \cite{MSS15} can be thought of as a random sampling result which produces a good frame with low but positive probability.  However, all known methods of producing frames which do stable phase retrieval using a number of vectors on the order of the dimension use sub-Gaussian random variables where a random sampling will produce a good frame with high probability.  In the following theorem we connect the problem of constructions of frames which do stable phase retrieval to the problem of discretizing the $L_1$-norm on a subspace of $L_1(\Omega)$.
In particular, we prove that in order to sample a continuous Parseval frame to obtain a frame which does stable phase retrieval, it is necessary to simultaneously discretize both the $L_1$-norm and the $L_2$-norm on the range of the analysis operator.

\begin{thm}\label{T:pr}
Let $(x_t)_{t\in\Omega}$ be a continuous Parseval frame for an $N$-dimensional real Hilbert space $H$ over a probability space $\Omega$ which does $\kappa$-stable phase retrieval and $\|x_{t}\|_{H}\leq \beta\sqrt{N}$ for all $t\in\Omega$. Let $T:H\rightarrow L_2(\Omega)$ be the analysis operator of $(x_t)_{t\in\Omega}$.  Suppose that $(t_j)_{j=1}^M\subseteq \Omega$ is such that $(\frac{1}{\sqrt{M}}x_{t_j})_{j=1}^M$ is a frame of $H$ with upper frame bound $B$ and lower frame bound $A$ which does $C$-stable phase retrieval.   Then both the $L_2$ norm and the $L_1$ norm on the range of the analysis operator are discretized in the following way for all $y\in T(H)$,
\begin{enumerate}
    \item $A\|y\|_{L_2(\Omega)}^2 \leq \dfrac{1}{M}\displaystyle \sum_{j=1}^M |y(t_j)|^2\leq B\|y\|_{L_2(\Omega)}^2$,
    \item $A^{1/2}B^{-3/2}C^{-3}(1+A^{-1}\beta^2)^{-3/2}\|y\|_{L_1(\Omega)} \leq \dfrac{1}{M}\displaystyle \sum_{j=1}^M |y(t_j)|\leq B^{1/2} \kappa^3(1+\beta^2)^{3/2}\|y\|_{L_1(\Omega)}$.
\end{enumerate}
\end{thm}

Before proving Theorem \ref{T:pr} we will prove the following lemma.

\begin{lem}\label{L:stabpr}
Let $\Omega$ be a probability space and let $X^N$ be an $N$-dimensional subspace of $ L_2(\Omega)$.  Suppose $\kappa,\beta>0$ are such that $\|x\|_{L_\infty}\leq \beta\sqrt{N}\|x\|_{L_2}$ for all $x\in X^N$ and that
\begin{equation}\label{E:stab_subspace}
\min(\|f-g\|_{L_2},\|f+g\|_{L_2})\leq \kappa \Big\||f|-|g|\Big\|_{L_2} \hspace{1cm}\textrm{ for all }f,g\in X^N.  
\end{equation}
 Then,
$\|x\|_{L_1}\leq \|x\|_{L_2}\leq \kappa^3(1+\beta^2)^{3/2} \| x\|_{L_1}$ for all $x\in X^N$.
\end{lem}

\begin{proof}
Let $x\in X^N$ with $\|x\|_{L_2}=1$. Note that $\|x\|_{L_1}\leq \|x\|_{L_2}$ as $\Omega$ is a probability space.
Let $\gamma=\|x\|_{L_1}^{1/3}$.
We have by Markov's inequality that
$$
\|x\|_{L_1}\geq \gamma Prob(|x|>\gamma).
$$
Hence, $Prob(|x|>\gamma)\leq \|x\|_{L_1}^{2/3}$.
Let $S=\{t\in \Omega:|x(t)|>\gamma\}$ and $P_S$ be the restriction operator from $L_2(\Omega)$ to $L_2(S)$.  Let $(e_j)_{j=1}^N$ be an orthonormal basis for $X^N$.  For each $t\in S$ there exists $\psi_t\in X$ such that $\langle x,\psi_t\rangle=x(t)$ for all $x\in X^N$.  Note that $\|\psi_t\|_{L_2}\leq \beta\sqrt{N}$ for all $t\in S$.  We have that
\begin{align*}
   \sum_{j=1}^N  \| P_S e_j\|^2_{L_2}&= \sum_{j=1}^N \int_S |e_j(t)|^2 dt\\
    &= \sum_{j=1}^N\int_{S} |\langle e_j,\psi_t\rangle|^2 dt\\
    &= \int_{S} \sum_{j=1}^N |\langle e_j,\psi_t\rangle|^2 dt\\
&= \int_{S} \|\psi_t\|^2 dt\\
&\leq Prob(S) \beta^2 N
\end{align*}
Thus, there exists $1\leq j\leq N$ such that $\|P_S e_j\|_{L_2}\leq (Prob(S))^{1/2} \beta$.  In particular, there exists $y\in X^N$ with $\|y\|_{L_2}=1$ and $\|P_S y\|_{L_2}\leq (Prob(S))^{1/2} \beta$.

Let $f=x+y$ and $g=x-y$.  As  $\|x\|_{L_2}=\|y\|_{L_2}=1$ we have that $\|f-g\|_{L_2}=\|f+g\|_{L_2}=2$.
We now obtain an upper bound for $\big\| |f|-|g|\big\|_{L_2}$.
\begin{align*}
\big\| |f|-|g|\big\|^2_{L_2}&=\big\| |x+y|-|x-y|\big\|^2_{L_2}\\
&=\int (2\min(|x(t)||y(t)|)^2 dt\\
&=4\int_{S} (\min(|x(t)||y(t)|)^2 dt+4\int_{S^c} (\min(|x(t)||y(t)|)^2 dt\\
&\leq 4\int_{S} |y(t)|^2 dt+4\int_{S^c} |x(t)|^2 dt\\
&\leq 4 \, Prob(S)\beta^2+4\gamma^2\quad \quad \quad \textrm{\small (as $\|P_S y\|^2_{L_2}\leq {Prob}(S) \beta^2$ and $|x(t)|\leq\gamma$ for all  $t\in S^c$).} \\
&\leq 4 \|x\|_{L_1}^{2/3}\beta^2+4\|x\|_{L_1}^{2/3}
\end{align*}
Thus, we have that 
\begin{equation}\label{E:L1}
\frac{1}{8}\Big\| |f|-|g|\Big\|^3_{L_2}\leq (1+\beta^2)^{3/2}\|x\|_{L_1}
\end{equation}
As $\|x\|_{L_2}=1$ and $\|f-g\|_{L_2}=\|f+g\|_{L_2}=2$ we have by \eqref{E:stab_subspace} and \eqref{E:L1} that 
$$\|x\|_{L_2}=\|x\|^{3}_{L_2}= \frac{1}{8}
\min(\|f-g\|^{3}_{L_2},\|f+g\|^{3}_{L_2})\leq \frac{1}{8}\kappa^3 \Big\||f|-|g|\Big\|_{L_2}^3
\leq \kappa^3 (1+\beta^2)^{3/2}\|x\|_{L_1}.
$$
\end{proof}

We now show that Theorem \ref{T:pr} follows from Lemma \ref{L:stabpr} 

\begin{proof}

Let $y=T(x)\in T(H)$ with $\|y\|_{L_2(\Omega)}=1$.  As $(x_t)_{t\in\Omega}$ is a Parseval frame for $H$, we have that $\|y\|_{L_2(\Omega)}=\|x\|_H=1$. 
Note that $y(t)=\langle x, x_{t}\rangle$ for all $t\in\Omega$.
As $(\frac{1}{\sqrt{M}}x_{t_j})_{j=1}^M$ is a frame of $H$ with lower frame bound $A$ and upper frame bound $B$, we have that
$$A\|y\|^2_{L_2(\Omega)}\leq\frac{1}{M} \sum_{j=1}^M |y(t_j)|^2\leq B\|y\|^2_{L_2(\Omega)}.
$$
We now have the following upper bound.
\begin{align*}
    \frac{1}{M}\sum_{j=1}^M |y(t_j)|&\leq \Big(\frac{1}{M}\sum_{j=1}^M |y(t_j)|^2\Big)^{1/2}\\
    &\leq B^{1/2} \|y\|_{L_2(\Omega)}\\
    &\leq B^{1/2} \kappa^3(1+\beta^2)^{3/2}\|y\|_{L_1(\Omega)}\hspace{1cm}\textrm{ \small by Lemma \ref{L:stabpr}}.
\end{align*}

We have that $(\frac{1}{\sqrt{M}}x_{t_j})_{j=1}^M$ is a frame with lower frame bound $A$ and upper frame bound $B$ which does $C$-stable phase retrieval.  Let the set $[M]=\{1,2,..,M\}$ be given the uniform probability measure.  Then, $(x_{t_j})_{j\in[M]}$ is a continuous frame with lower frame bound $A$ and upper frame bound $B$ which does $C$-stable phase retrieval.  Let $T_{[M]}:H\rightarrow L_2([M])$ be the analysis operator of $(x_{t_j})_{j\in[M]}$. Thus, we have for all $f,g\in H$ that $$\min_{|\lambda|=1}\|T_{[M]}f-\lambda T_{[M]}g\|_{L_2([M])}\leq B^{1/2} \min_{|\lambda|=1}\|f-\lambda g\|_{H}\leq B^{1/2}C\big\||T_{[M]}f|-|T_{[M]}g|\big\|_{L_2([M])}.$$
Furthermore, as $\|x_{t}\|_{H}\leq \beta N^{1/2}$ for all $t\in\Omega$, we have for all $f\in H$ that
$$\|T_{[M]}f\|_{L_\infty([M])}=\sup_{j\in[M]} |\langle f,x_{t_j}\rangle|\leq \|f\|_H \beta N^{1/2}\leq A^{-1/2} \beta N^{1/2}\|T_{[M]} f\|_{L_2([M])} .
$$

We can thus apply Lemma \ref{L:stabpr} to the subspace $T_{[M]}H\subseteq L_2(\Omega)$ with stability constant $B^{1/2}C$ and $L_\infty$ bound $A^{-1/2} \beta$ to calculate the following.
\begin{align*}
    \frac{1}{M}\sum_{j=1}^M |y(t_j)|&=\|y\|_{L_1([M])}\\
    &\geq B^{-3/2}C^{-3}(1+A^{-1}\beta^2)^{-3/2} \|y\|_{L_2([M])}\hspace{1cm}\textrm{ \small by Lemma \ref{L:stabpr}}\\
    &\geq A^{1/2}B^{-3/2}C^{-3}(1+A^{-1}\beta^2)^{-3/2}  \|y\|_{L_2(\Omega)}\\
 &\geq A^{1/2}B^{-3/2}C^{-3}(1+A^{-1}\beta^2)^{-3/2}  \|y\|_{L_1(\Omega)}
\end{align*}
\end{proof}

\section{Discretizing infinite dimensional subspaces of $L_p$}\label{S:infinite}

In sections \ref{S:L1} and \ref{S:Lp} we showed that Theorem \ref{T:LT21} does not hold for finite dimensional subspaces of $L_p[0,1]$ for $1\leq p<2$.  We now show that Corollary \ref{C:disc} fails in a strong way for infinite dimensional subspaces of $L_p(\R)$ for $1\leq p<2$.

\begin{prop}
For all $1\leq p<2$ there exists a subspace $Y\subseteq L_p(\R)$ such that $Y$ is isomorphic to $\ell_p$ and $\|x\|_{L_\infty}\leq \|x\|_{L_p}$ for all $x\in Y$, and the following holds. If $J\subseteq \R$ is such that $\sup_{t\in J}|x(t)|\neq 0$ for all $x\in Y\setminus\{0\}$ then there exists $y\in Y$ such that 
$\sum_{t\in J}|y(t)|^p=\infty$.
\end{prop}

\begin{proof}

We first consider the case $1<p<2$.   By Proposition \ref{P:bad} we have for all $N\in\N$ that there exists a subspace $X^N\subseteq L_p[0,1]$ such that $X^N$ is $A_p^{-1}B_p$-isomorphic to $\ell_2^N$ and the following hold,
\begin{enumerate}
    \item $\|x\|_{L_\infty}\leq A^{-1}_p N^{1/p}\|x\|_{L_p}$ for all $x\in X^N$,\\
    \item If $(t_j)_{j=1}^M\subseteq[0,1]$ are such that $(\sum_{j=1}^M |x(t_j)|^p)^{1/p}>0$ for all $x\in X^N\setminus\{0\}$ then 
    $$ N^{2-p/2}\|x\|_{L_p}^{p}\leq \sum_{j=1}^M |x(t_j)|^p\hspace{1cm}\textrm{ for some }x\in X^N.$$
\end{enumerate}
Let $(M_N)_{N=1}^\infty$ be an increasing sequence of real numbers such that $M_{N+1}>M_N+A^{-p}_pN$ for all $N\in\N$.
For each $N\in\N$, we let $D_N:X^N\rightarrow L_2([M_N, M_N+A_p^{-p} N])$ be the operator defined by for $x\in X^N$,
$$(D_N x)(t)=A_p N^{-1/p} x\big(A_p ^{p}N^{-1}(t-M_N)\big)\hspace{1cm}\textrm{ for all }t\in [M_N, M_N+A_p^{-p} N].$$
Note that $D_N$ is an isometric embedding of $X^{N}$ into $L_2([M_N ,M_N+ A_p^{-p} N])$.  We let $Y^N=D_N(X^{N})$. By (1) we have that $\|y\|_{L_\infty}\leq \|y\|_{L_p}$ for all $y\in Y^N$.  Let $J_N\subseteq\R$ such that 
  $(\sum_{t\in J_N} |y(t)|^p)^{1/p}>0$ for all $y\in Y^N\setminus\{0\}$.  Let $I_N=(J_N-M_N)A_p^{p}N^{-1} $.  By (2), there exists $x_N\in X^N$ with $\|x_N\|^p_{L_p}=N^{p/2-2}$ and $\sum_{t\in I_N} |x_N(t)|^p\geq 1$.  We let $y_N=D_N x_N$ and hence $\|y_N\|^p_{L_p}=N^{p/2-2}$ and $\sum_{t\in J_N} |y_N(t)|^p\geq A_p^{p} N^{-1}$. 

We now let $Y=\overline{\spn}(Y^N)_{N\in\N}\subseteq L_p(\R)$.  As each $Y^N$ is $A_p^{-1}B_p$-isomorphic to $\ell_2^N$ and is composed of functions which are supported on the interval $[M_N,M_N+ A_p^{-p} N]$, we have that $Y$ is isomorphic to $(\oplus \ell_2^N)_{\ell_p}$ and hence $Y$ is isomorphic to $\ell_p$ by Pełczyński's decomposition theorem. For all $y\in Y$, we have that $\|y\|_{L_\infty}\leq \|y\|_{L_p}$.  As, $\|y_N\|^p_{L_p}=N^{p/2-2}$ and $1< p<2$, we have that $\sum_{N=1}^\infty \|y_N\|^p<\infty$.  Thus, $\sum_{N=1}^\infty y_N\in Y$.  We now suppose that $J\subseteq \R$ is such that $\sup_{t\in J}|y(t)|\neq 0$ for all $y\in Y\setminus\{0\}$.  As $(y_N)_{N=1}^\infty$ have pairwise disjoint support, we  have that
$$\sum_{t\in J}\Big|\sum_{N=1}^\infty y_N(t)\Big|^p=\sum_{N=1}^\infty\sum_{t\in J}|y_N(t)|^p\geq\sum_{N=1}^\infty A_p^{p} N^{-1}=\infty.
$$

Thus, we have completed the proof for the case $1<p<2$.  Note that the exact same proof does not work for $p=1$ as $(\oplus \ell_2^n)_{\ell_1}$ is not isomorphic to $\ell_1$.  However, if we instead use Theorem \ref{T:L1} instead of Proposition \ref{P:bad} then the proof follows in the same way.
\end{proof}

\end{document}